\documentclass[11pt]{amsart}

\usepackage{amssymb}
\usepackage{amsthm}
\usepackage{amsmath}
\usepackage{mathrsfs}
\usepackage{mathtools}
\usepackage[bbgreekl]{mathbbol}
\usepackage{amsbsy}
\usepackage{wasysym}
\usepackage{stmaryrd}
\usepackage{hyperref}
\usepackage{tikz}
\usepackage{array}
\usepackage{enumerate}
\usepackage{hhline}
\usepackage[style=alphabetic,maxbibnames=99,maxalphanames=99]{biblatex}
\usepackage{graphicx}
\usepackage{color}

\usepackage{geometry}
\usepackage[noabbrev,capitalise]{cleveref}

\newenvironment{enumerate*}%
{\begin{enumerate}[(I)]%
\setlength{\itemsep}{10pt}%
\setlength{\parskip}{0pt}}%
{\end{enumerate}}

\newtheorem{Thm}{Theorem}[section]
\newtheorem{Prop}[Thm]{Proposition}
\newtheorem{Cor}[Thm]{Corollary}
\newtheorem{Conj}[Thm]{Conjecture}
\newtheorem{Ques}[Thm]{Question}

\newtheorem{Lm}[Thm]{Lemma}

\theoremstyle{definition}
\newtheorem{Df}[Thm]{Definition}
\newtheorem{Rmk}[Thm]{Remark}
\newtheorem{Ex}[Thm]{Example}

\newcommand{\R}{\ensuremath{\mathbb R}} 
\newcommand{\C}{\ensuremath{\mathbb C}} 
\newcommand{\G}{\ensuremath{\mathbb G}} 
\newcommand{\deq}{\ensuremath{\coloneqq}}
\newcommand{\pa}[1]{\left(#1\right)}

\DeclareMathOperator{\Hex}{Hex}
\DeclareMathOperator{\Aut}{Aut}

\DeclareMathOperator{\Hom}{Hom}
\DeclareMathOperator{\Pb}{\mathbb{P}}
\DeclareMathOperator{\Hb}{\mathbb{H}}
\DeclareMathOperator{\K}{\mathbb{K}}

\DeclareMathOperator{\F}{\mathbb{F}}
\DeclareMathOperator{\eps}{\epsilon}
\DeclareMathOperator{\bp}{\boxplus}
\DeclareMathOperator{\Nc}{\mathcal{N}}
\DeclareMathOperator{\Fb}{\mathbb{F}}
\DeclareMathOperator{\Pib}{\mathbb{\Pi}}

\title{On the asymptotic behavior of finite hyperfields}
\author{Tuong Le, Chayim Lowen}

\address[]{Department of Mathematics, Princeton University, Princeton, NJ 08540, USA}
\email{tl0101@princeton.edu}

\address[]{Department of Mathematics, Princeton University, Princeton, NJ 08540, USA}
\email{chayiml@princeton.edu}

\begin{document}
\begin{abstract}
    Hobby has recently shown that almost all finite hyperfields of even order fail to be the quotient of a field. Using a probabilistic argument, we extend this result to all orders: a finite hyperfield is almost always non-quotient. This confirms a conjecture of Baker--Jin. We show that in almost every finite hyperfield the sum of any four or more nonzero elements contains 0. We also give a precise asymptotic for the number of finite hyperfields on a given finite abelian group.
\end{abstract}

\maketitle

\section{Introduction and summary of results}\label{sec:intro}
Hyperfields are a generalization of fields obtained by relaxing the requirement that addition be single-valued.
These were introduced by Krasner \cite{Krasner56} as a means of studying valued fields. 
Recently, hyperfields have received renewed interest in light of the Baker--Bowler theory of matroids over hyperfields (see \cite{baker17})---a framework which simultaneously generalizes matroids, oriented matroids and valuated matroids and has seen applications in the theory of Lorentzian polynomials (\cite{triangular}, \cite{huang}).

In \cite{Krasner1983}, Krasner 
gave a method---generalizing the one in \cite{Krasner56}---to construct hyperfields as quotients of fields by multiplicative subgroups. He then asked if every hyperfield arises in this way.
A negative answer was given in \cite{Massouros} where a general schema for constructing counterexamples is provided. In \cite{Jin}, Baker and Jin give an algorithm for determining which finite hyperfields are quotients of finite fields and, additionally, a necessary criterion for a hyperfield to be the quotient of an infinite field. Based on their results, they conjectured that almost all finite hyperfields are non-quotients:
\begin{Conj}[{\cite[Question (4)]{Jin}}]\label[Conj]{conj:100}
    Let $\mathcal{H}_n$ be the set of isomorphism classes of hyperfields of order $n$, and let $\mathcal{Q}_n \subseteq \mathcal{H}_n$ be the subset of those which are isomorphic to a quotient of some field. Then
    \[
        \lim_{n \to \infty} \frac{\#\mathcal{Q}_n}{\#\mathcal{H}_n} = 0.
    \]
\end{Conj}
In \cite{hobby}, Hobby showed that \Cref{conj:100} holds if we restrict our attention to hyperfields of even order $n$.
Our main result is a full proof of this conjecture. We in fact prove the following stronger result.
\begin{Thm}\label[Thm]{thm:main}
        For a finite abelian group $G$ and $\epsilon \in G$ of order at most $2$, 
        let $\mathcal{H}(G, \epsilon)$ be the set of hyperfields with underlying group $G$ in which $-1$ is given by $\epsilon$---up to isomorphism preserving $\epsilon$. Let also $\mathcal{I}(G, \epsilon)$ be the subset of these that are homomorphic images of skew fields. Then\footnote{Our asymptotic notation follows Knuth's widely-used conventions set forth in \cite{Knuth}.
        Importantly, the expression $f = \Omega(g)$ for us means that there exists a \emph{positive} constant $C > 0$ such that 
$f(n) \geq Cg(n)$ for all sufficiently large $n$. This differs markedly from the definition in Hardy and Littlewood \cite[\textsection 2.21]{Hardy}. 
}
        \[
             \frac{\#\mathcal{I}(G, \epsilon)}{\#\mathcal{H}(G, \epsilon)} = e^{-\Omega(\#G)}.
        \]
\end{Thm}
\begin{Cor}\label[Cor]{cor:100}
    Asymptotically almost no finite hyperfield is isomorphic to the quotient of a skew field.\footnote{
    Hyperfields are commutative by definition, but they may be quotients of non-commutative skew fields. See \Cref{rmk:commutativeskewquotient} for an example.}
\end{Cor}
We go further and identify a relatively simple class of hyperfields which asymptotically comprises almost all finite hyperfields. 
\begin{Df}\label[Df]{df:4-full}
        A \emph{4-full} hyperfield is a hyperfield $\mathbb{H} \neq \mathbb{F}_2$ in which
        \[
            0 \in a \boxplus b \boxplus c \boxplus d
        \]
        holds for all $a, b, c, d \in \mathbb{H}^{\times}$.
\end{Df}
    \begin{Rmk}
        A 4-full hyperfield $\mathbb{H}$ has the property that 
        $0 \in a_1 \boxplus \cdots \boxplus a_m$ for all $a_1, \dots, a_m \in \mathbb{H}^\times$ whenever $m \geq 4$. This is one reason for excluding the field $\mathbb{F}_2$.
    \end{Rmk}
\begin{Df}\label[Df]{df:0/0}
    A hyperfield $\mathbb{H}$ has the 0/0 property if for all
    $x \in \mathbb{H}$ we can find $r, s \in 1 \boxplus -1$ with $s \neq 0$ such that $x = rs^{-1}$. Informally, we have ``$\frac{1-1}{1-1} = \mathbb{H}$''---hence the name.
\end{Df}
\begin{Rmk}
    The $0/0$ hyperfields are precisely the ones for which (categorical) products can naturally be constructed. See \Cref{sec:appendix} for a precise statement and proof.
\end{Rmk}
    \begin{Thm}\label[Thm]{thm:runner_up}
        For a finite abelian group $G$ and $\epsilon \in G$ of order at most $2$, let $\mathcal{H}(G, \epsilon)$ be the set of hyperfields with underlying group $G$ in which $-1$ is given by $\epsilon$---up to isomorphism preserving $\epsilon$. Let also $\mathcal{F}(G, \epsilon)$ be the subset of these that are both 4-full and $0/0$. Then         
        \[
             \frac{\#\mathcal{F}(G, \epsilon)}   {\#\mathcal{H}(G, \epsilon)} = 1 - e^{-\Omega(\# G)}.
        \]
    \end{Thm}
    \begin{Cor}\label[Cor]{cor:4-full}
    Asymptotically almost all finite hyperfields are 4-full and 0/0.
\end{Cor}
    It follows from \cite[Theorem 4.1]{fusion} that every 4-full hyperfield is perfect\footnote{This notion of perfection is unrelated to the identically-named notion in field theory. It is instead related to the well-behavedness of the theory of matroids over the given hyperfield.} in the sense of \cite{baker19}. Hence \Cref{cor:4-full} shows:
    \begin{Cor}
         Almost all finite hyperfields are perfect.
    \end{Cor}
    In \cite[Question 5.13]{rank}, Baker, Solomon and Zhang asked if every hyperfield has the so-called FETVINS property. This means that every system of $m$ homogeneous linear equations in $m+1$ unknowns
    has a nonzero solution.
    The proof of \cite[Theorem 5.2]{hobby} shows that every 4-full hyperfield with the 0/0 property also has the FETVINS property.\footnote{More precisely, in his Lemma 5.4, Hobby shows that what he calls \emph{ample} hyperfields have the property labeled ($\star$) in our \Cref{prop:4fullpasture} and the proof of his Theorem 5.4 shows that ($\star$) implies FETVINS. In \Cref{prop:4fullpasture}, we show that ($\star$) is equivalent to the conjunction of the 4-fullness and 0/0 properties.}
    Hence \Cref{cor:4-full} shows:
    \begin{Cor}
          Almost all finite hyperfields have the FETVINS property.
    \end{Cor}
    It is a well-known principle that sufficiently rich combinatorial structures should generically have no automorphisms. This is indeed the case for finite hyperfields.
    \begin{Thm}\label[Thm]{thm:noauto}
        Let $G$ be a finite abelian group and $\epsilon \in G$ of order at most 2. Let 
        $\mathcal{H}(G, \epsilon)$ be the set of hyperfields with underlying group $G$ in which $-1$ is given by $\epsilon$---up to isomorphism preserving $\epsilon$. Let $\mathcal{A}(G, \epsilon)$ be the subset of these that have some non-trivial automorphism. Then
        \[
             \frac{\#\mathcal{A}(G, \epsilon)}{\#\mathcal{H}(G, \epsilon)} = e^{-\Omega(\# G^2)}.
        \]
    \end{Thm}
    \begin{Cor}
        Almost all finite hyperfields have no non-identity automorphisms.
    \end{Cor}
    Finally, helped by \Cref{thm:noauto}, we give an asymptotic formula for the number of isomorphism classes of hyperfields with a given abelian group $G$ as its multiplicative group.
    \begin{Thm}\label[Thm]{thm:count}
        Let $G$ be a sufficiently large\footnote{The requirement that $G$ be large is included to ensure that the expression $\frac{\#G[2]}{\#\Aut G}\cdot  2^{\frac{1}{6}(\#G^2 + 3\#G + 2\#G[3])}$ is an \emph{overestimate} for the number of hyperfields with underlying group $G$.} finite abelian group. Let 
        $\mathcal{H}(G)$ be the set of isomorphism classes of hyperfields with underlying group $G$. Then 
        \[
            \# \mathcal{H}(G)  = \left(1- e^{-\Theta(\#G)}\right) \frac{\#G[2]}{\# \Aut G}\cdot  2^{\frac{1}{6}(\#G^2 + 3\#G + 2\#G[3])}
        \]
        where $G[2]$ and $G[3]$ are, respectively, the 2-torsion and 3-torsion subgroups of $G$.\footnote{Here we employ Knuth's $\Theta$-notation: 
    $f = \Theta(g)$ is the conjunction of 
    $f = \Omega(g)$ and $g = \Omega(f)$.}
    \end{Thm}
    \begin{Cor}\label[Cor]{cor:asymptotic}
        Let $\mathcal{H}_n$ be the set of isomorphism classes of hyperfields of order $n$. Then 
        \[
            \log_2 \#\mathcal{H}_n = n^2/6 + O(n)\,.
        \]
    \end{Cor}
    The most important ingredient in the proof of these results is the simple but useful notion of \emph{hexagons} of \emph{fundamental pairs} in a hyperfield.
    These notions first arose in Semple's thesis \cite{semple}.
    They were made more explicit in the later work of Pendavingh and Van Zwam \cite{pendavingh10} in the setting of partial fields. Baker and Lorscheid \cite{baker25} applied these ideas in the more general setting of pastures---of which hyperfields may be seen as a special case---and also coined the evocative term ``hexagon''.
    Their framework is essentially equivalent to the notion of \emph{blocks} employed by Hobby.
    We revert to the earlier Baker--Lorscheid nomenclature, with subtle changes in meaning (see \Cref{rmk:deviations}).

    \subsection*{Overview of the paper.}
    In \cref{sec:notation}, we define hyperfields, hexagons and fundamental pairs and explain the relationships between them. In \cref{sec:lottery}, we lay out the probabilistic framework with which we prove the main results. We then use it to prove the basic properties of random finite hyperfields.
    In \Cref{sec:results}, we use the previous work to deduce the results announced above. We conclude in \Cref{sec:questions} by commenting on the possibility of extending our results to the setting of finite skew hyperfields.
    In \Cref{sec:appendix}, we demonstrate the intimate connection between the $0/0$ property and product construction for hyperfields.
    In \Cref{sec:skew}, we prove that finite hyperfields predominate among finite skew hyperfields.
\section{Hyperfields, pastures and hexagons}\label{sec:notation}
In this section we discuss the necessary background material for our results. It will be convenient at first to work in the slightly more general setting of skew hyperfields.
For this and other reasons, we will always write groups---abelian or not---multiplicatively and denote by $1$ their identity elements. 
\subsection*{A word about multivalued binary operators} Given a set $S$, a \emph{multivalued} binary operator is a function $\star: S \times S \to \mathcal{P}(S)$ where $\mathcal{P}(S)$ is the power set of $S$. 
Following the usual conventions, given $s, t \in S$ we write 
$s \star t$ for $\star(s, t)$.
Given such a binary operator, we may naturally extend it to a binary operator on $\mathcal{P}(S)$ via
\[
    A \star B = \bigcup_{\substack{a \in A\\b \in B}} a \star b
\]
For $s \in S$ and $A \subseteq S$, we also write $s \star A$ for $\{s\} \star A$ and $A \star s$ for $A \star \{s\}$.
We say that $\star$ is \emph{commutative} if $s \star t = t \star s$ for all $s,t \in S$. This holds if and only if the single-valued operator $\star : \mathcal{P}(S) \times \mathcal{P}(S) \to \mathcal{P}(S)$ is commutative.
Similarly, we say that $\star$ is \emph{associative} if $(s \star t) \star u = s \star (t \star u)$ for all $s, t, u \in S$. Again, this holds if and only if $\star : \mathcal{P}(S) \times \mathcal{P}(S) \to \mathcal{P}(S)$ is associative. As is the case for single-valued associative operations, if $\star$ is associative we may write expressions such as 
$x_1 \star x_2 \star \dots \star x_n$ without ambiguity---where now each $x_i$ is either an element or a subset of $S$.\footnote{As is the usual practice in all areas of mathematics---with the exception of set theory---we assume implicitly that the set $S$ under consideration is disjoint from its powerset, so no ambiguity arises.}
\begin{Df}\label[Df]{df:hyperfield}
    A \emph{skew hyperfield} is a sextuple $\mathbb{H} = (G, \cdot, \boxplus, 0, 1, -1)$ in which
    \begin{itemize}
        \item $(G, \cdot, 1)$ is a group, called the \emph{underlying group} of $\mathbb{H}$. We write $\mathbb{H}^{\times}$ for $G$.
        \item $-1 \in G$ is an element in the center of $G$ satisfying $(-1)^2 = 1$, called the \emph{unit} of $\mathbb{H}$.
        \item $0$ is a formally-adjoined symbol not already contained in $G$. We call $0$ the \emph{zero} of $\mathbb{H}$. We identify $\mathbb{H}$ with the set $G \sqcup \{0\}$. We formally extend the multiplication rule on $\mathbb{H}^{\times}$ by setting 
        $0 \cdot 0 = 0\cdot g = g \cdot 0 = 0$ for all $g \in G$.\footnote{The reader will observe that multiplication is then an associative binary operator on $\Hb$ with identity element $1$ for which the invertible elements are precisely those in $\Hb^{\times}$. It is commutative if $G$ is abelian.}
        \item $\boxplus: \mathbb{H} \times \mathbb{H} \to \mathcal{P}(\mathbb{H})$ is a multivalued binary operator, called \emph{addition}.
    \end{itemize}
    such that these elements together satisfy
    \begin{enumerate}[(I)]
        \item The operator $\boxplus$ is associative and commutative.
        \item For any $g, h \in \mathbb{H}$, we have 
        $0 \in g \boxplus h$ if and only if $g = -1 \cdot  h$.
        \item For all $g, h, k \in \mathbb{H}$, we have $g(h \boxplus k) = gh \boxplus gk$ and 
        $(h \boxplus k)g = hg \boxplus kg$.
    \end{enumerate}
    We say that $\mathbb{H}$ is a \emph{hyperfield} if $G$ is abelian.
\end{Df}
\begin{Rmk}
    We write $-g$ for $g \in \mathbb{H}$ to mean $-1 \cdot g$. Note that $-(-g) = g$, as expected.
\end{Rmk}
\begin{Rmk}\label[Rmk]{rmk:zero}
    It is usually added as an axiom that $g \boxplus 0 = 0 \boxplus g = \{g\}$ for all $g \in \mathbb{H}$. In our setup, this follows from the commutativity of $\boxplus$ and the equivalences
    \[
    x \in g \boxplus 0 \;\overset{\text{(II)}} {\Leftrightarrow}\;
    0 \in - x \boxplus (g \boxplus 0)
    \;\overset{\text{(I)}} {\Leftrightarrow}\;
    0 \in (- x \boxplus g) \boxplus 0
    \;\overset{\text{(II)}} {\Leftrightarrow}\;
   - x \boxplus g \ni -1 \cdot 0 = 0
   \;\overset{\text{(II)}} {\Leftrightarrow}\;
   - g = - x
   \; {\Leftrightarrow}\;
   g = x
    \]
    Similarly, the rule that
    $g \in h \boxplus k$ if and only if 
    $h \in g \boxplus - k$ for 
    $g,h,k \in \mathbb{H}$ follows from
    \[
        g \in h \boxplus k 
        \;\overset{\text{(II)}} {\Leftrightarrow}\;
        0 \in - g \boxplus (h \boxplus k)
         \;\overset{\text{(I)}} {\Leftrightarrow}\;
          0 \in h \boxplus (- g \boxplus k)
        \;\overset{\text{(III)}} {\Leftrightarrow}\;
        0 \in - h \boxplus (g \boxplus - k)
        \;\overset{\text{(II)}} {\Leftrightarrow}\;
        h \in g \boxplus - k
    \]
\end{Rmk}
\begin{Rmk}\label[Rmk]{rmk:nonempty}
    Although we allowed $\boxplus$ to spit out any subset of $\mathbb{H}$, in fact we can never have $g \boxplus h = \varnothing$. For then by applying (II) and then (I) we would get
    \[
        \{h\} = 0 \boxplus h \subseteq (- g \boxplus g) \boxplus h = - g \boxplus (g \boxplus h) = -g \boxplus \varnothing = \varnothing
    \]
\end{Rmk}
\begin{Rmk}\label[Rmk]{rmk:mindata}
    Because of \Cref{rmk:zero}, when specifying a skew hyperfield, it is not necessary to specify addition with 0. Similarly, because of rule (III), it is only necessary to specify 
    $1 \boxplus g$ for all $g \in \mathbb{H}^{\times}$. The value of 
    $g \boxplus h$ for $g,h \in \mathbb{H}^{\times}$ is deduced via 
    $g \boxplus h = g(1 \boxplus g^{-1}h)$.
    Finally, the value of $-1$ can be deduced from the fact that 
    $0 \in 1 \boxplus g$ if and only if $g = -1$.
\end{Rmk}
\begin{Ex}\label[Ex]{ex:field=hyper}
    Any skew field $\mathbb{F}$ gives rise to a canonical skew hyperfield with underlying group $\mathbb{F}^{\times}$ and addition defined by
    $a \boxplus b \deq \{a+b\}$.
    It is a hyperfield if and only if $\mathbb{F}$ is a field.
    The skew hyperfields $\mathbb{F}$ that arise this way are precisely the ones in which $\boxplus$ is single-valued. That is,
    $\#(a \boxplus b) = 1$ for all $a, b \in \mathbb{F}$.
\end{Ex}
\begin{Rmk}\label[Rmk]{rmk:check1-1}
    To check that $\boxplus$ is single-valued in a (skew) hyperfield, it in fact suffices to verify that $1 \boxplus -1 = \{0\}$. Indeed, if this is the case, then for any $a,b \in \mathbb{F}$ and any $x,y \in a \boxplus b$
    \[ 
        x \boxplus - y \subseteq (a \boxplus b) \boxplus -(a \boxplus b) 
        = (a \boxplus - a) \boxplus (b \boxplus - b)
        = a(1 \boxplus -1) \boxplus 
        b(1\boxplus -1)
        = \{0\} \boxplus \{0\} = \{0\}
    \]
    This implies $x = y$ by \Cref{rmk:nonempty}
    and axiom (II).
\end{Rmk}
\begin{Ex}\label[Ex]{ex:Krasner}
    If we take $G$ to be the trivial group and set $1 \boxplus 1 = \{0, 1\}$, the resulting hyperfield is called the \emph{Krasner hyperfield} and is denoted $\mathbb{K}$.
\end{Ex}
\begin{Ex}\label[Ex]{ex:sign}
    If we take $G = \{1, -1\}$---the cyclic group of order 2---and define
    $1 \boxplus 1 = \{1\}$, 
    $1 \boxplus -1 = \{1, 0, -1\}$, the resulting hyperfield is called the \emph{sign hyperfield} and is denoted $\mathbb{S}$.
\end{Ex}
\begin{Ex}
    If we take $G = \mathbb{R}_{>0}$, identify $G \sqcup \{0\}$ with $\mathbb{R}_{\geq 0}$, and define addition by
    \[
        x \boxplus y \deq \begin{cases}
            \{\max(x, y)\} & \text{if $x \neq y$}\\
            [0, x] & \text{if $x = y$}
        \end{cases}
    \]
    then the resulting hyperfield is called the \emph{tropical hyperfield} and is denoted $\mathbb{T}_0$.
\end{Ex}
\begin{Ex}\label[Ex]{ex:triangle}
    If we take $G = \R_{>0}$, identify $G \sqcup \{0\}$ with $\R_{\geq0}$, and define
    \[
        x \boxplus y = \{z \in \R_{\geq 0} \mid \text{$x,y,z$ form the sidelengths of a possibly degenerate triangle}\}
    \]
    then the resulting hyperfield is called the \emph{triangular hyperfield} and is denoted 
    $\mathbb{T}_1$.
\end{Ex}
\begin{Df}
    A homomorphism between skew hyperfields $\mathbb{H}$ and $\mathbb{G}$ is a group homomorphism $f: \mathbb{H}^{\times} \to \mathbb{G}^{\times}$ which---when formally extended by setting $f(0) = 0$---satisfies
    \[
        f(g \boxplus h) \subseteq f(g) \boxplus f(h)
    \]
    for all $g, h \in \mathbb{H}^{\times}$. This gives rise to a category of skew hyperfields. The procedure outlined in \Cref{ex:field=hyper} to reinterpret skew fields as skew hyperfields induces a fully faithful embedding of the category of skew fields into the category of skew hyperfields. Clearly, the subcategory of fields lands inside the subcategory of hyperfields.
\end{Df}
\begin{Rmk} 
    If $f: \mathbb{H} \to \mathbb{G}$ is a skew hyperfield homomorphism,
    then with the convention $f(0) = 0$ we have
    $f(gh) = f(g)f(h)$
    for all $g, h \in \mathbb{H}$.
    Also, $f(-1) = -1$.
\end{Rmk}
Skew hyperfields are of interest primarily for their interaction with skew fields. For instance, a homomorphism $\mathbb{F} \to \mathbb{S}$ from a skew field $\mathbb{F}$ is the same data as a total ordering of $\mathbb{F}$, a homomorphism 
$\mathbb{F} \to \mathbb{T}_0$ is the same as the data of a (real-valued) valuation on $\mathbb{F}$,
and a homomorphism 
$\mathbb{F} \to \mathbb{T}_1$ is the same as the data of an absolute value on $\mathbb{F}$. A more direct interaction between skew fields and skew hyperfields is given by the following construction. 
\begin{Df}
    Let $\mathbb{F}$ be a skew field and let $\Gamma \subseteq \mathbb{F}^{\times}$ be a normal subgroup of its multiplicative group. The group
    $\mathbb{F}^{\times}\!/\Gamma$ comes equipped with a canonical skew hyperfield structure (see \cite{Krasner1983}). Namely, for 
    $\overline{a}, \overline{b} \in \mathbb{F}^{\times}\!/\Gamma \cup \{0\}$, we set
    \[
        \bar a \boxplus \bar b = \left\{\overline{\alpha + \beta} \mid \alpha, \beta \in \mathbb{F}, \bar{\alpha} = \bar a, \bar{\beta} = \bar b\right\}
    \]
    where $\bar x$ denotes the coset 
    $x \Gamma$ if $x \in \mathbb{F}^{\times}$ and $\bar 0 = 0$.
    The resulting hyperfield is denoted $\mathbb{F}\!/\Gamma$ and comes equipped with a canonical \emph{quotient homomorphism}
    $\mathbb{F} \to \mathbb{F}\!/\Gamma$.
    A skew hyperfield is called \emph{quotient} if it is (isomorphic to a skew hyperfield) of this form.
\end{Df}
\begin{Rmk}
    The quotient construction can be defined more generally whenever $\mathbb{F}$ is a skew \emph{hyperfield}. However, we will not have need for this level of generality.
\end{Rmk}
\begin{Ex}
    The sign hyperfield $\mathbb{S}$ of \Cref{ex:sign} is the quotient $\mathbb{R}/\mathbb{R}_{>0}$.
\end{Ex}
\begin{Ex}\label[Ex]{ex:Kquot}
    The Krasner hyperfield $\mathbb{K}$ of \Cref{ex:Krasner} is the quotient $\mathbb{F}\!/\mathbb{F}^{\times}$ for any skew field $\mathbb{F}$ of cardinality at least 3.
\end{Ex}
\begin{Ex}\label[Ex]{ex:trianglequot}
    The triangular hyperfield 
    $\mathbb{T}_1$ of \Cref{ex:triangle} is the quotient 
    $\mathbb{C}/S^1$ where $S^1$ is the group of complex numbers of modulus 1.
\end{Ex}
\begin{Rmk}\label[Rmk]{rmk:commutativeskewquotient}
    Every quotient of a field is a hyperfield. However, a non-commutative skew-field may have a hyperfield quotient. For instance, if $\mathbb{H}$ is the real quaternion algebra then 
    $\mathbb{H}/[\mathbb{H}^{\times}, \mathbb{H}^{\times}]$ is a hyperfield isomorphic to
    $\mathbb{T}_1$.
    By \Cref{ex:trianglequot}, $\mathbb{T}_1$ is also a quotient of $\C$.
    The present authors were unable to construct an example of a hyperfield which is a quotient of some \emph{non-commutative} skew field, but not of any field (c.f.\ \Cref{ques:skew}).
\end{Rmk}
We now begin our analysis of hyperfields specifically.
We will describe them in terms of the triples $(a, b, c)$ such that 
$a \bp b \bp c \ni 0$. We begin with the following definition.
\begin{Df}\label[Df]{df:hexagon}
    Let $G$ be an abelian group. There are natural actions of $G$ and $S_3$ on the set $G^3$, the former by multiplication and the latter by permutation of the coordinates.
    A \emph{fundamental pair} of $G$ is a $G$-orbit of $G^3$. We identify the orbit of the element $(x, y, z) \in G^3$ with the pair 
    $(xz^{-1}, yz^{-1})$---hence the name.
    This gives an identification of the set of fundamental pairs with $G^2$.
    
    A \emph{hexagon} of $G$ is an 
    $(S_3 \times G)$-orbit of $G^3$. 
    We write $\overline{(x,y,z)}$ for the $(S_3 \times G)$-orbit of $(x,y,z) \in G^3$.
    We identify the hexagon of $(x, y, z) \in G^3$ with the multiset of the six corresponding fundamental pairs. These are
    \[
        (xz^{-1}, yz^{-1}),\, (yz^{-1}, xz^{-1}),\, 
        (xy^{-1}, zy^{-1}),\, (zy^{-1}, xy^{-1}),\,
        (yx^{-1}, zx^{-1}),\, (zx^{-1}, yx^{-1})
    \]
    In terms of the fundamental pair $(u, v) \deq (xz^{-1}, yz^{-1})$, these six are
    \[
        (u, v),\, (v, u),\, 
        (uv^{-1}, v^{-1}),\, (v^{-1}, uv^{-1}),\,
        (vu^{-1}, u^{-1}),\, (u^{-1}, vu^{-1})
    \]
    The group $S_3$ acts on the set of all fundamental pairs of $G$, with each orbit being a hexagon. 
    We write $\hexagon(G)$ for the set of hexagons of $G$
    and $\overline{(u, v)}$ for the hexagon containing the fundamental pair $(u, v)$.
    We may visualize these data via a hexagonal diagram as in \Cref{fig:hexagon}.
    \begin{figure}[h]
        \centering
        \def\svgwidth{0.8\textwidth}
        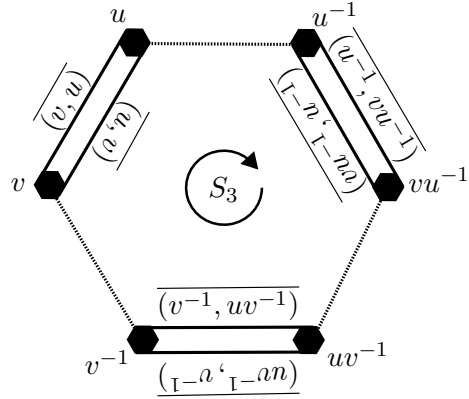
        \caption{A hexagon of fundamental pairs}
        \label{fig:hexagon}
    \end{figure}
\end{Df}
It follows from axioms (I) and (III) in \Cref{df:hyperfield} that for a hyperfield $\mathbb{H}$ the set 
\[
    \{(a,b,c) \in \mathbb{H}^{\times} \mid 0 \in a \boxplus b \boxplus c\}
\]
is invariant under both the multiplicative action of 
$(\mathbb{H}^{\times})^3$ and the permutation action of $S_3$. It is therefore natural to consider the corresponding set of hexagons in $\hexagon(\mathbb{H}^{\times})$.

It is not too difficult to see from the axioms (I)-(III) that a hyperfield is determined by the data of (i) its underlying group, (ii) its unit and (iii) its collection of hexagons---that is, the set of all $(x,y,z)\in G^3$ such that $x\boxplus y\boxplus z\ni 0$ (see \Cref{prop:hexagon_characterization}). Not every collection of such data gives rise to a hyperfield, however.
It is therefore worth giving a name to the simpler structure consisting of such data. This was done by Baker and Lorscheid in \cite{baker21}, under the banner ``pastures''.
Pastures were used to great effect in \cite{foundations} to systematize and generalize a number of results in the representation theory of matroids.
We now define pastures based on the framework above. For comparison, see \cite[Definition 2.1]{foundations}.
\begin{Df}\label[Df]{df:pasture}
    A \emph{pasture} $\mathbb{P}$ consists of a triple $\mathbb{P} = (G, -1, \Hex(\mathbb{P}))$ where 
    \begin{itemize}
        \item $G$ is an abelian group, called the \emph{underlying group} of $\mathbb{P}$ and denoted 
        $\mathbb{P}^{\times}$.
        \item $-1 \in G$ is an element satisfying $(-1)^2 = 1$, called the \emph{unit} of $\mathbb{P}$.
        \item $\Hex(\mathbb{P}) \subseteq \hexagon(G)$ is a subset of hexagons of $G$, called the \emph{(nontrivial) 
        nullset} of $\mathbb{P}$.
    \end{itemize}
\end{Df}
\begin{Rmk}\label[Rmk]{rmk:deviations}
    We deviate from the nomenclature in \cite{baker21} in two respects. Firstly, in the Baker--Lorscheid conventions, a fundamental pair of $(x,y,z)$ is represented by 
    the pair $(- xz^{-1}, - yz^{-1})$ instead of $(xz^{-1}, yz^{-1})$. Our convention has the advantage that the set of available hexagons in a group $G$ is independent of the unit $-1$. 
    Secondly, the \emph{nullset} in the sense of Baker and Lorscheid is---using our terminology---the set
    \[
        \{ 0 + 0 + 0 \} \cup \{x + \epsilon x + 0 \mid x \in G\} \cup \left\{x + y + z \;\middle|\; \text{$\overline{(x, y, z)}$ lies in $\Hex(\Pb)$}\right\}
    \]
    where $a + b + c$ is a formal symbol representing the $S_3$-orbit of $(a,b,c)$ in $(\mathbb{P}^{\times} \sqcup \{0\})^3$. Our present conventions will make subsequent arguments more salient.
\end{Rmk}
As we have already observed, to a hyperfield $\Hb$ we can associate the pasture consisting of the underlying group $G= \Hb^\times$,  the unit $-1$ of $\mathbb{H}$, and the hexagons of $\Hb$. 
Below we give a characterization of those pastures which give rise to hyperfields.
\begin{Prop}\label[Prop]{prop:hexagon_characterization}
    The pasture associated to a hyperfield uniquely determines it.
    A pasture $\Pb$ on a group $G$ arises from a hyperfield if and only if the following two conditions hold:
    \begin{enumerate}[(A)]
        \item For all $x \in G$ with 
        $x \neq -1$, there exists $y \in G$
        such that $\overline{(x,y)} \in \Hex(\mathbb{P})$.
        \item For all $x,y,z,w \in G$, if 
        $\overline{(x, y)} \in \Hex(\Pb)$, 
        $\overline{(z, w)} \in \Hex(\Pb)$, and $(x, y) \neq (z, w)$ then there exists $t \in G$ such that $\overline{(tx, -tz)} \in \Hex(\Pb)$
        and $\overline{(-ty, tw)} \in \Hex(\Pb).$
    \end{enumerate}
\end{Prop}
\begin{proof}  
    Let $\widehat G \deq G \sqcup \{0\}$.
    If $\mathbb{P}$ is the pasture of some hyperfield, then the addition operator $\boxplus$ can be recovered via the rule
    \[
        x \boxplus y \deq \begin{cases}
            \{y\} & \text{if $x = 0$}\\
            \{x\} & \text{if $y = 0$}\\
            \left\{z \in G \;\middle|\;\overline{(x,y,-z)} \in \Hex(\mathbb{P})\right\} & \text{if $x, y \in G$, $x \neq - y$}\\
            \left\{z \in G \;\middle|\;\overline{(x,y,-z)} \in \Hex(\mathbb{P})\right\} \cup \{0\} & \text{if $x, y \in G$, $x = - y$}\\
        \end{cases}
    \]
    On the other hand---for any pasture $\mathbb{P}$ on $G$---if we 
    \emph{define} a binary operator
    $\boxplus: \widehat G \times \widehat G \to \mathcal{P}(\widehat G)$ by this rule, then $\boxplus$ is a well-defined, commutative, multivalued binary operator. It satisfies axiom (II) of \Cref{df:hyperfield} by construction. We verify that it satisfies axiom (III). Note that
    $g(h \boxplus k) = gh \boxplus gk$ holds trivially if $g = 0$ and follows from $x \boxplus 0 = 0 \boxplus x = \{x\}$ if $h = 0$ or $k = 0$. We may therefore assume $g,h,k \in G$. Then, by construction, for $z \in G$
    \begin{align*}
        z \in gh \boxplus gk 
         \;&\Leftrightarrow\;
         \overline{(gh, gk, -z)} \in \Hex(\mathbb{P})\\
        \;&\Leftrightarrow\;
        \overline{(h, k, -g^{-1}z)} \in \Hex(\mathbb{P})
        \\&\Leftrightarrow\;
        g^{-1}z \in h \boxplus k
        \;\Leftrightarrow\;
        z \in g(h \boxplus k)
    \end{align*}
    Since $G$ is abelian, this suffices to establish (III). It remains now to determine the associativity of $\boxplus$. Note that by the commutativity of $\boxplus$, associativity is reduced to the claim
    \[
        \text{for all $g, h, k \in G$,} \quad (g \boxplus h) \boxplus k \subseteq g \boxplus (h \boxplus k)
    \]
    We show that this holds if and only if $\mathbb{P}$ satisfies conditions (A) and (B). The identity $(g \boxplus h) \boxplus k = g \boxplus (h \boxplus k)$ holds trivially if any of $g,h,k$ is $0$. Suppose 
    $g,h,k \neq 0$. By the validity of (II) we have
    \[
        0 \in (g \boxplus h) \boxplus k
        \;\Leftrightarrow\;
         - k \in g \boxplus h
        \;\Leftrightarrow\;
        \overline{(g, h, k)} \in \Hex(\mathbb{P})
        \;\Leftrightarrow\;
        - g \in h \boxplus k
        \;\Leftrightarrow\;
        0 \in g \boxplus (h \boxplus k)
    \]
    From this case it follows that for $z \in G$ we have
    \[
    z \in (g \boxplus h) \boxplus k 
    \Leftrightarrow\; 0 \in ((g \boxplus h) \boxplus k) \boxplus -z\;\Leftrightarrow\;
    0 \in (g \boxplus h) \boxplus (k \boxplus -z)
    \]
    Similarly, $z \in g \boxplus (h \boxplus k)$ if and only if 
    $0 \in (- z \boxplus g) \boxplus (h \boxplus k)$.
    Setting $\ell = -z$, the associativity condition reads
    \begin{equation}\label{eq:assoc}
        \text{for all $g, h, k, \ell \in G$, \quad if}
        \quad 
        0 \in (g \boxplus h) \boxplus (k \boxplus \ell) \quad\text{then}\quad
        0 \in (g \boxplus \ell) \boxplus (h \boxplus k)
    \end{equation}
    Suppose $0 \in (g \boxplus h) \boxplus (k \boxplus \ell)$ holds. This means precisely that one of the following is true.
    \begin{itemize}
        \item $0 \in g \boxplus h$ and $0 \in k \boxplus \ell$.
        \item $s \in g \boxplus h$ and 
        $- s \in k \boxplus \ell$ for some $s \in G$.
    \end{itemize}
    We show that the validity of (\ref{eq:assoc}) in the first case is precisely condition (A) and its validity in the second case is precisely condition (B).

    In the first case, $h = - g$
    and $k = - \ell$. The desired conclusion is
    \[
        0 \in (g \boxplus \ell) \boxplus (- g \boxplus - \ell)
        = (g \boxplus \ell) \boxplus - (g \boxplus \ell)
    \]
    This holds if and only if $g \boxplus \ell \neq \varnothing$.
    This is certainly true if $g = - \ell$. Otherwise, it is equivalent to the existence of $q \in G$ such that $\overline{(g, \ell, q)} \in \Hex(\mathbb{P})$, i.e.\ $\overline{(g\ell^{-1}, q\ell^{-1})} \in \Hex(\mathbb{P})$.
    Since $g$ and $\ell$ are arbitrary, this reduces precisely to condition (A).

    In the second case, the hypothesis on $g,h,k,\ell$ is
    that there exists $s \in G$ such that $\overline{(g, h, s)} \in \Hex(\mathbb{P})$
    and $\overline{(k, \ell, -s)} \in \Hex(\mathbb{P})$. Because of axiom (III), the condition in question is invariant under the multiplicative action of $G$. Thus, after multiplying $g,h,k,\ell$ by $s^{-1}$, the hypothesis becomes 
    $\overline{(g, h, 1)}, \overline{(k, \ell, -1)} \in \Hex(\mathbb{P})$. That is,
    $\overline{(g, h)}, \overline{(-k, -\ell)} \in \Hex(\mathbb{P})$.

    If $g = -\ell$ and $h = -k$, then the conclusion of (\ref{eq:assoc}) holds trivially.
    If $g \neq -\ell$ or $h \neq -k$, the conclusion of (\ref{eq:assoc}) is the condition that there exists $r \in G$ such that $\overline{(g, \ell, r)} \in \Hex(\mathbb{P})$
    and $\overline{(h, k, -r)} \in \Hex(\mathbb{P})$.
    We rewrite this as
    \[
        \overline{(r^{-1} g, r^{-1}\ell)} \in \Hex(\mathbb{P})
        \quad \text{and} \quad
        \overline{(-r^{-1}h, -r^{-1}k)} \in \Hex(\mathbb{P})
    \]
    Using the correspondence $x \leftrightsquigarrow g$, $y \leftrightsquigarrow h$, 
    $z \leftrightsquigarrow -\ell$, 
    $w \leftrightsquigarrow -k$, $t \leftrightsquigarrow r^{-1}$, we recover precisely condition (B).
\end{proof}
\begin{Rmk}\label[Rmk]{rmk:norestriction}
    The restriction $x\neq -1$ in condition (A) can be removed precisely when the hyperfield is not a field. This is a restatement of \Cref{rmk:check1-1}.
\end{Rmk}
The following result uses the notions of \emph{4-full} 
and \emph{0/0} hyperfields as defined in \Cref{sec:intro}
(Definitions \ref{df:4-full} and \ref{df:0/0}).
\begin{Prop}\label[Prop]{prop:4fullpasture}
    A pasture $\Pb$ gives rise to a 4-full, 0/0 hyperfield if and only if
    \begin{equation}\tag{$\star$}\label{eq:starstar}
    \text{For all $x, y, z, w \in G$, there exists $t \in G$ such that both
            $\overline{(tx, ty)},\overline{(tz, tw)}\in \Hex(\Pb)$.}
    \end{equation}
\end{Prop}
\begin{proof}
    Suppose $\Pb$ satisfies \eqref{eq:starstar}. For any $x\in G$, applying \eqref{eq:starstar} to $(1,x,1,x)$ shows there exist $t$ such that $(t, tx)\in \Hex(\Pb)$ which implies $(t^{-1},x)\in \Hex(\Pb)$. Thus $\mathbb{P}$ satisfies condition $(A)$ in \Cref{prop:hexagon_characterization}.
    Condition $(B)$ is trivially satisfied as well. So $\Pb$ gives rise to a hyperfield.
    
    Letting $s = t^{-1}$, property (\ref{eq:starstar}) can be translated into the language of hyperfield as follows: 
    \begin{itemize}
        \item[($\star\star$)] For all $x,y,z,w\in G$, there exist $s\in G$ such that $s\bp x\bp y\ni 0$ and $s\bp z\bp w\ni 0$.
    \end{itemize}
    It therefore suffices to show that a hyperfield satisfies $(\star\star)$ if and only if it is $4$-full and $0/0.$ We will show that a hyperfield is $4$-full if and only if $(\star\star)$ holds whenever $y \neq -x, z \neq -w$, and is $0/0$ if and only if $(\star\star)$ holds in the special case $y = -x, z = -w$.

    For the first half, note that $0\in x\bp y\bp z\bp w$ if and only if there exists $s\in \Hb$ such that $s\in x\bp y$ and $-s\in z\bp w.$
    If $y = -x$ and $z = -w$ this condition is trivially satisfied by taking $s = 0$. Otherwise, 
    $s$ cannot be taken to be $0$, and this condition is the same as: $-x\bp -y\bp s\ni 0$ and $s\bp w\bp z\ni 0.$ That is, it is equivalent to $(\star\star)$ for $(-x,-y,z,w).$ Since $x\mapsto -x$ is an involution on $G$, we have shown that $4$-fullness is the same as $(\star\star)$ being valid except possibly when $y = -x, z = -w.$

    We now prove the remaining claim that the $0/0$ property is the same as the validity of
    $(\star\star)$  for all quadruples $(x,-x,z,-z)$ with  $x,z\in G$. Indeed, note that $x\bp -x\bp t\ni 0$ is the same as $-1 \boxplus 1 \boxplus -x^{-1} t\ni 0$, which is the same as  $x^{-1}t\in 1\bp -1.$ Thus, letting $S = 1\bp -1$, the content of $(\star\star)$ for $(x,-x, z, -z)$ reads: there exists $t$ such that $x^{-1}t\in S$ and $z^{-1}t\in S.$ This is true if and only if $zx^{-1}$ can be expressed as $rs^{-1}$ for $r, s\in S.$ The claim therefore follows from the fact that the map $(x,z)\mapsto zx^{-1}$ from $G^2$ to $G$ is surjective.
\end{proof}
It will be useful to define morphisms between pastures. This was already done, for instance, in \cite{baker25}. We express the definition using our conventions.
\begin{Df}\label[Df]{df:pasturemorphism}
    A \emph{morphism} of pastures from 
    a pasture $\mathbb{P}_1$ to 
    a pasture $\mathbb{P}_2$ is a group homomorphism 
    ${f: \mathbb{P}_1^{\times} \to \mathbb{P}_2^{\times}}$ mapping the unit of $\mathbb{P}_1$ to the unit of $
    \mathbb{P}_2$ and such that 
    \[
        \text{if \quad $\overline{(x,y)} \in \Hex(\mathbb{P}_1)$ \quad then \quad $\overline{(f(x), f(y))} \in \Hex(\mathbb{P}_2)$}
    \]
\end{Df}
We omit the easy proof of the following proposition.
\begin{Prop}\label[Prop]{prop:fullyfaithful}
    A homomorphism of hyperfields gives rise to a morphism of the corresponding pastures. Every morphism between these pastures arises in this way. This defines a fully faithful embedding of the category of hyperfields into the category of pastures.
\end{Prop}
\section{A hyperfield lottery}\label{sec:lottery}
In this section, we fix a finite abelian group $G$ of cardinality $n \deq \#G$ and a ``unit element'' $\eps\in G$ satisfying $\eps^2 =1$.
We will consider the process of constructing a random pasture on $G$ with unit $\epsilon$ by selecting a subset $\Nc\subseteq \hexagon(G)$ uniformly at random and thereby obtaining the pasture $\Pb = (G, \eps, \Nc)$.
This gives rise to a pasture-valued random variable ${\Pib}$. We call it the \emph{hyperfield lottery}, for reasons that will soon become clear.
Note that the selection of $\mathcal{N}$ may be described in the following equivalent way: each element of $\hexagon(G)$ is included in $\mathcal{N}$ with probability $\frac{1}{2}$, independently of all other hexagons. It is straightforward to translate from the probabilistic language to a counting statement: $\Pib$ satisfies a condition $X$ with probability $p$ if and only if 
the number of pastures on $(G, \epsilon)$ with property $X$ is 
$p \cdot 2^{\#\hexagon(G)}$. Nevertheless, for reasons that may already be apparent, it will be very useful to frame matters probabilistically.

Below, we shall establish a number of properties of the hyperfield lottery. The first shows that it almost always returns a hyperfield.
Specifically, we give a lower bound on the probability that $\Pib$ satisfies condition (\ref{eq:starstar}) of \Cref{prop:4fullpasture}.

\begin{Prop}\label[Prop]{prop:prob}
        For $n$ sufficiently large, the probability that $\Pib$ satisfies (\ref{eq:starstar}) is
        at least $1 - e^{-cn}$ for some absolute constant $c > 0$.
    \begin{proof}
        It suffices to show that for some absolute constant $c > 0$ and any $x, y, z, w \in G$
        \begin{itemize}
            \item[$\boldsymbol{\rightarrow}$] 
            With probability at least $1 - e^{-cn}$, $\overline{(tx, ty)} \in \mathcal{N}$
            and $\overline{(tz, tw)} \in \mathcal{N}$ hold simultaneously for at least one $t \in G$.
        \end{itemize}
        Indeed, if this is the case for all $(x,y,z,w) \in G^4$, then (\ref{eq:starstar}) will hold with probability at least 
        $1 - n^4 e^{-cn}$. This is bounded below by $1 - e^{-c'n}$ for any $0 < c' < c$ and sufficiently large $n$.
        
        We now fix $x, y, z, w$ and consider the sequence of (unordered) pairs of hexagons
        \[
            h_t \deq \left\{\overline{(tx, ty)}, \overline{(tz, tw)}\right\}, \qquad (t \in G)
        \]
        (Note that $h_t$ may degenerate to contain only one element.) 
         
         The tuples 
        $(tx, ty)$ for $t \in G$ are all distinct. The same is true for $(tz, tw)$. Since a hexagon contains at most 6 elements, each hexagon of $(G, \epsilon)$ can occur in at most 12 different hexagonal pairs $h_t$.
        We form a graph $\Gamma$ 
        with vertex set $G$
        by drawing an edge between
        $s, t \in G$ whenever the hexagonal pairs $h_s, h_t$ intersect in a common hexagon. Then $\Gamma$ has $n$ vertices and each vertex of $\Gamma$ has degree at most $2\times(12 - 1) = 22$. By a standard graph-theoretic argument, this guarantees that $\Gamma$ contains an independent set of size 
        $\geq n / 23$. That is, we can find $t_1, \dots, t_m \in G$ 
        with $m = \lceil n / 23 \rceil$
        such that 
        $h_{t_1}, \dots, h_{t_m}$ 
        are pairwise disjoint. For any $i$, the probability that the one or two hexagons in 
        $h_{t_i}$ lie in $\Nc$ is at least $\left(\frac{1}{2}\right)^2 = 1/4$.
        Recall that in choosing $\Nc$, the events that this or that hexagon is included are all jointly independent.
        Since the hexagonal pairs $h_{t_1}, \dots, h_{t_m}$ are pairwise disjoint,
        \[
            \mathbb{P}\left[h_{t_1} \nsubseteq \Nc \wedge \dots \wedge
            h_{t_m} \nsubseteq \Nc\right] = 
            \prod_{i = 1}^m \mathbb{P}\left[h_{t_i} \nsubseteq \Nc\right] \leq \left(1 - \frac{1}{4}\right)^m \leq \left(\frac{3}{4}\right)^{n/23}
        \]
        If $h_{t_i} \subseteq \Nc$ for some $i$, then the statement above is true for $(x, y, z, w)$.
        We may therefore take 
        $c = \frac{1}{23}\log(4/3)$.
    \end{proof}
 \end{Prop}

 The above proposition gives an upper bound for the probability that $\Nc$ is \emph{not} a hyperfield. The following result provides a lower bound on this probability.

\begin{Prop}\label[Prop]{prop:nonhyper}
    In the hyperfield lottery, the probability that $\Pib$ is \emph{not} a hyperfield is at least 
    $2^{-n}$, provided $n \geq 2$.
\end{Prop}
\begin{proof}
    Let $g \in G \setminus \{ \epsilon\}$.
    The probability that the nullset of $\Pib$ does not contain any hexagons of the form
    $\overline{(g,x)}$ is at least $2^{-n}$, since there at most $n$ such hexagons. When this happens $\Pib$ cannot be a hyperfield, since it fails condition (A) of \Cref{prop:hexagon_characterization}.
\end{proof}
 
\begin{Prop}\label[Prop]{prop:G-G}
     The probability that $\overline{(\epsilon, x)} \in \Nc$ for all $x \in G$ is at most $e^{-cn}$ for some absolute constant $c > 0$.
 \end{Prop}
    \begin{proof}
    As $x$ ranges over $G$, the expression $\overline{(\epsilon, x)}$ takes on at least $n/6$ distinct values. Let 
        $x_1, \dots, x_m \in G$ be such that 
        $\overline{(\epsilon, x_1)}, \dots, \overline{(\epsilon, x_m)}$ are distinct with $m = \lceil n/6 \rceil$. The events that one or another of $\overline{(\epsilon, x_i)}$ is included in $\mathcal{N}$ are jointly independent, each occurring with probability $\frac{1}{2}$. Hence
        \[
            \mathbb{P}\left[\overline{(\epsilon, x_1)} \in \mathcal{N} \wedge \dots \wedge \overline{(\epsilon, x_m)} \in \mathcal{N}\right] = \prod_{i = 1}^m \mathbb{P}\left[\overline{(\epsilon, x_i)} \in \mathcal{N}\right] = \left(\frac{1}{2}\right)^m \leq \left(\frac{1}{2}\right)^{n/6}
        \]
        We may therefore take $c = \frac{1}{6} \log 2$.
    \end{proof}
 
    Before discussing the next property, we need the following elementary observations.
    \begin{Lm}\label[Lm]{lm:numhex}
        Let $\mathbb{H}$ be any hyperfield on $G$. Then $\mathbb{H}$ has at least $(n - 1)/6$ hexagons.
    \end{Lm}
    \begin{proof}
       Consider any element $x \neq 0, -1$ in $\mathbb{H}$. By \Cref{rmk:nonempty}, the set $x \boxplus 1$ contains some nonzero element $y$. Then 
       $x \boxplus -y \ni -1$. Thus $(x, -y)$ is a fundamental pair. It follows that each of the $n-1$ elements of $\mathbb{H} \setminus \{0, -1\}$ lies in a fundamental pair of some hexagon of $\mathbb{H}$. Since each hexagon can account for at most 6 elements, we get the promised bound.
    \end{proof}
    
    \begin{Lm}\label[Lm]{lm:numend}
        There are at most $n^{\log_2 n}$ endomorphisms $G \to G$.        
    \end{Lm}
    \begin{proof}
        By the structure theorem for finite abelian groups, we can write $G = \prod_{i=1}^m G_i$ where each $G_i$ is a non-trivial cyclic group. Since a homomorphism out of a cyclic group is determined by the image of a generator, we have
        \[
            \#\Hom(G, G) = \#\prod_{i = 1}^m \Hom(G_i, G) = \prod_{i = 1}^m  \#\Hom(G_i, G) \leq \prod_{i = 1}^m \#G = n^m
        \]
        On the other hand,
        \[
            n = \#G = \#\prod_{i=1}^m G_i = \prod_{i =1}^m \#G_i \geq \prod_{i = 1}^m 2 = 2^m
        \]
        So $m \leq \log_2 n$ and $\#\Hom(G,G) \leq n^{\log_2 n}$.
    \end{proof}
    \begin{Rmk}
        More precise upper bounds can be obtained---in the generality of all finite groups. See for instance \cite{Neumann95}. The crude upper bound above will suffice for our purposes.
    \end{Rmk}
    \begin{Prop}\label[Prop]{prop:surjmorphism}
        Let $\mathbb{H}$ be a hyperfield on $G$.
        For sufficiently large $n$, the probability that there exists a bijective morphism $\mathbb{H} \to \Pib$ is at most $e^{-cn}$ for some absolute constant $c > 0$.
    \end{Prop}
    \begin{proof}
        There are at most $\#\Aut G$ possibilities for the underlying group homomorphism ${G \to G}$. Since, by \Cref{lm:numend},
        \[
            \#\Aut G \leq \#\Hom(G, G) \leq n^{\log_2 n} \leq e^{(\log_2 n)^2} = e^{o(n)}
        \]
        it suffices to show that the \emph{identity} map $G \to G$ induces a morphism $\Hb \to \Pib$ with probability at most $e^{-cn}$ for some absolute constant $c > 0$. If $\epsilon$ is not the unit of $\mathbb{H}$, then this simply cannot occur. Otherwise, it happens if and only if 
        $\Hex(\Hb) \subseteq \mathcal{N}$.
        Now $\Hex(\Hb)$ contains at least ${(n-1)/6}$ hexagons, by \Cref{lm:numhex}.
        Since the hexagons in $\Pib$ are chosen to be included independently---each with probability $\frac{1}{2}$---the probability that $\mathcal{N}$ contains all of them is at most
        $2^{-(n - 1)/6}$. We may therefore take 
        $c = \frac{1}{7}\log 2$.
    \end{proof}
    Though we have not needed it earlier, it will at this point be extremely beneficial to have a good estimate for the size of $\hexagon(G)$.
    \begin{Prop}\label[Prop]{prop:count}
        Write $G[3]$ for the 3-torsion subgroup of $G$. Then
        \[
            \# {\hexagon}(G) = \frac{n^2 + 3n + 2\#G[3]}{6}
        \]
    \end{Prop}
    \begin{proof}
        We appeal to Burnside's Lemma for the group $S_3$ acting on the fundamental pairs in $G^2 = G^3/G$.
        The number of fundamental pairs fixed under the identity is of course $\#G^2$. The elements fixed under the transposition $(x, y) \mapsto (y, x)$ are exactly the pairs $(x, x)$ for $x \in G$. There are exactly $n$ of these. By symmetry, the other transpositions have the same number of fixed-points. The 3-cycle $(x, y) \mapsto (yx^{-1}, x^{-1})$ has as its fixed-points precisely the pairs $(x, x^2)$ where $x \in G$ satisfies $x^3 = 1$.
        There are $\#G[3]$ such elements.
        The inverse 3-cycle $x,y \mapsto (y^{-1}, xy^{-1})$ has the same number of fixed-points. Thus the average number of fixed points is precisely the number stated above.
    \end{proof}
    From \Cref{prop:count}, we see that $\#\hexagon(G) \sim\frac{\#G^2}{6}$. We now use this to show that most finite pastures have no automorphisms.
    \begin{Prop}\label[Prop]{prop:noauto}
        For sufficiently large $n$, the probability that $\Pib$ admits a non-trivial automorphism is at most $2^{-cn^2}$ for some absolute constant $c > 0$.
    \end{Prop}
    \begin{proof}
        By \Cref{lm:numend}, $G$ has at most $e^{o(n)}$ automorphisms. Therefore, it suffices to show that for any fixed non-trivial automorphism $f: G \to G$, the probability that $\mathcal{N}$ is invariant under $f$ is at most $2^{-c'n^2}$ for some absolute constant $c'$.
        We first show that a positive proportion of hexagons in $\hexagon(G)$ fail to be fixed by $f$. We write $[x : y : z]$ for the orbit of $(x,y,z) \in G^3$ under $G$, which is a fundamental pair. If the hexagon in which $[x : y : z]$ lies is fixed by $f$ then
        \[
            [f(x) : f(y) : f(z)] =
            [x:y:z]^{\sigma} \quad \text{for some $\sigma \in S_3$}
        \]
        If $\sigma$ is the identity, then
        \[
            [f(x) : f(y) : f(z)] = [x : y :z]
        \]
        shows that $xz^{-1}, yz^{-1} \in G^f$.
        There are precisely $\#\hexagon(G^f)$ such hexagons.
        
        If $\sigma$ is a transposition, then---without loss of generality---we may assume 
        $\sigma = (1 2)$. Then
        \[
            [f(x) : f(y) : f(z)] = [y : x : z]
        \]
        shows that $yz^{-1} = f(xz^{-1})$. Thus the hexagon $\overline{(x,y,z)}$ has the form $\overline{(g, f(g))}$. There can be at most $\#G$ such hexagons for a given transposition $\sigma$, so at most $3\#G$ overall.

        If $\sigma$ is a 3-cycle, then---without loss of generality---we may assume 
        \[
            [f(x) : f(y) : f(z)] = [z:x:y]
        \]
        In particular, this gives 
        $yz^{-1} = f(zx^{-1})$.
        That is,
        $yz^{-1} = f(xz^{-1})^{-1}$. So the hexagon $\overline{(x,y,z)}$ has the form $\overline{(g, f(g)^{-1})}$. Again, there can be at most $\#G$ such hexagons for a given 3-cycle $\sigma$, so $2\#G$ overall.

        Since $f$ is a non-trivial automorphism, we have $\#G^f \leq \#G/2$.
        Using \Cref{prop:count}, it follows that the number of hexagons fixed by $f$ is at most
        \[
            \#\hexagon(G^f) + 5\#G \leq \frac{n^2}{24} + O(n) = \left(\frac{1}{4} + o(1)\right)\#\hexagon(G)
        \]
        So for large $n$, at most a third of the hexagons of $G$ are fixed by $f$. Hence the number of $f$-orbits of $\hexagon(G)$ is at most 
        $\frac{2}{3}\#\hexagon(G)$. This means that there are at most $2^{\frac{2}{3}\#\hexagon(G)}$ pastures over $G$ that are invariant under $f$ compared to the $2^{\#\hexagon(G)}$ pastures in total. 
        The result now follows from the asymptotic
        $\#\hexagon(G) \sim \frac{n^2}{6}$.
    \end{proof}
    We end this section with a general lemma that transforms statements about the hyperfield lottery to statements about hyperfields considered up to isomorphism.
    \begin{Lm}\label[Lm]{lm:modiso}
        There exists an absolute constant $c > 0$ such that the following holds.
        Let $S$ be any collection of isomorphism classes of hyperfields on $G$ with unit $\epsilon$. Let $p$ be the probability that $\Pib$ lies in $S$. Then with the notation of \Cref{thm:main} we have
        \[
            \frac{\#S}{\#\mathcal{H}(G, \epsilon)} \leq cp e^{(\log_2 n)^2}
        \]
        as soon as $\#G$ is sufficiently large.
    \end{Lm}
    \begin{proof}
        By \Cref{lm:numend}, the group $\mathscr{A} \deq \Aut(G, \epsilon)$ of automorphisms of $G$ fixing $\epsilon$ has order at most $e^{(\log_2 n)^2}$. 
        Since each element of $\mathcal{H}(G, \epsilon)$ corresponds to an $\mathscr{A}$-orbit of $2^{\hexagon(G)}$---and since each such orbit has size at most $e^{(\log_2 n)^2}$---the number of elements in $2^{\hexagon(G)}$ that give rise to a hyperfield is at most $e^{(\log_2 n)^2} \cdot \#\mathcal{H}(G, \epsilon)$.
        On the other hand, it is at least 
        $(1 - e^{-c'n})\cdot 2^{\#\hexagon(G)}$ for some absolute constant $c' >0$, by \Cref{prop:prob}. Hence
        \[ 
            e^{(\log_2 n)^2} \cdot \#\mathcal{H}(G, \epsilon) \geq (1 - e^{-c'n})\cdot 2^{\#\hexagon(G)}
        \]
        Now since any class in $S$ is realized by \emph{some} nullset in $2^{\hexagon(G)}$, we have 
        \[
        p \geq \frac{\#S}{2^{\#\hexagon(G)}}
        \]
        Multiplying the two inequalities above and rearranging, we get
        \[
            \frac{e^{(\log_2 n)^2}}{1 - e^{-c'n}} \cdot p \geq \frac{\#S}{\#\mathcal{H}(G, \epsilon)} 
        \]
        We can therefore take $c = (1 - e^{-c'})^{-1}$.
    \end{proof}
\section{Proofs of results}\label{sec:results}
\begin{proof}[Proof of \Cref{thm:runner_up}]
    Using the hyperfield-lottery terminology of \Cref{sec:lottery}, the probability that $\Pib$ is a hyperfield which is either not 0/0 or not 4-full is $e^{-\Omega(\#G)}$, by \Cref{prop:prob}. Hence as a direct consequence of \Cref{lm:modiso}, we have
    \[
        1-\frac{\#\mathcal{F}(G, \epsilon)}{\#\mathcal{H}(G, \epsilon)} = O\left(e^{-\Omega(\#G)} \cdot e^{(\log_2 \#G)^2}\right) = e^{-\Omega(\#G)} \qedhere
    \]
\end{proof}
\begin{Rmk}
    The proof of \Cref{thm:runner_up} also shows that asymptotically almost all finite pastures are hyperfields. This is an a priori surprising statement since pastures are ``less structured'' than hyperfields.
\end{Rmk}
As was observed in \cite{Jin}, the following result of Turnwald \cite{Turnwald} provides a strong constraint for small quotients of large fields. (See also the earlier result of Bergelson–Shapiro \cite{Bergelson}, which proves the result for infinite fields.)
\begin{Thm}[{\cite[Theorem 1]{Turnwald}}]\label[Thm]{thm:Turnwald}
    Let $\mathbb{F}$ be a skew field and let $\Gamma \subseteq \mathbb{F}^{\times}$ be a proper subgroup of finite index in $\mathbb{F}^{\times}$.
    If $\#\Fb^{\times} > (\Fb^{\times}:\Gamma)^4$ then $\Gamma - \Gamma = \Fb$.
\end{Thm}
\begin{Rmk}
    Turnwald actually gives the bound $\#\F \geq (k-1)^4 + 4k$ where ${k = (\mathbb{F}^{\times} : \Gamma)}$.
    This gives a stronger statement for $k \geq 2$.
    For $k = 1$, the modified statement is easy to verify.
    The cleaner expression will be more convenient to work with.
\end{Rmk}
\begin{Cor}\label[Cor]{cor:1-1}
    Let $G$ be a finite group of order $n$. Up to isomorphism, there are at most 
    $n^4$ {quotient} skew hyperfields $\mathbb{H}$ with underlying group $G$ in which 
    $1 \boxplus -1 \neq \mathbb{H}$.
\end{Cor}
\begin{proof}
    Suppose $\mathbb{H}$ is the quotient of some skew field $\mathbb{F}$ by a normal subgroup $\Gamma \subseteq \mathbb{F}^{\times}$, necessarily of finite index. Then 
    $1 \boxplus -1$ is the image of 
    $\Gamma - \Gamma$ under the quotient homomorphism. Also $n =
    (\mathbb{F}^{\times} : \Gamma)$.
    Suppose
    $1 \boxplus -1 \neq \mathbb{H}$.
    By \Cref{thm:Turnwald}, $\mathbb{F}^{\times}$ can have cardinality at most $n^4$. Hence $\mathbb{F}$ is a finite field. For any given finite field $\mathbb{F}$, the cyclic group $\mathbb{F}^{\times}$ has at most one quotient of order 
    $n$. So---up to isomorphism---there is at most one hyperfield on $G$ which is a quotient of $\mathbb{F}$.
    Since there is at most one finite field of any given order, there are at most $n^4$ possibilities for $\mathbb{H}$.
\end{proof}
\begin{proof}[Proof of \Cref{thm:main}]
    We consider again the random pasture $\Pib$ of \Cref{sec:lottery}. 
    Any morphism $\phi: \mathbb{F} \to \Pib$ from a skew field $\mathbb{F}$ factors as $\mathbb{F} \to \mathbb{F}/\ker \phi \to \Pib$.
    Hence $\mathbb{\Pib}$ admits a surjective morphism from a skew field if and only if it admits a bijective homomorphism from a quotient hyperfield $\mathbb{\Theta}$ on $G$. If in $\mathbb{\Theta}$ we have $1 \boxplus -1 = \mathbb{\Theta}$, then in $\mathbb{\Pib}$ we must have
    $1 \boxplus -1 = \mathbb{\Pib}$. The latter is precisely the statement that $\overline{(-1, x)} \in \Hex(\mathbb{\Pib})$ for all $x \in G$. By \Cref{prop:G-G}, this happens with probability $e^{-\Omega(\#G)}$. 
    On the other hand, if in $\mathbb{\Theta}$ we have $1 \boxplus -1 \neq \mathbb{\Theta}$, then by \Cref{cor:1-1}, $\mathbb{\Theta}$ is one of at most $\#G^4$ quotient hyperfields on $G$ with that property. For each of these hyperfields $\mathbb{\Theta}$, 
    a bijective morphism $\mathbb{\Theta} \to \Pib$ exists with probability $e^{-\Omega(\#G)}$, by \Cref{prop:surjmorphism}.
    The probability that this holds for \emph{some} such $\mathbb{\Theta}$ is at most $\#G^4 \cdot e^{-\Omega(\#G)} = e^{-\Omega(\#G)}$.
    Hence the probability that $\Pib$ is a hyperfield which admits a surjective homomorphism from a field is at most $e^{-\Omega(\#G)}$. 
    By \Cref{lm:modiso}, this implies
    \[
        \frac{\#\mathcal{I}(G, \epsilon)}{\#\mathcal{H}(G, \epsilon)} = O\left(e^{-\Omega(\#G)} \cdot e^{(\log_2 n)^2}\right) = e^{-\Omega(\#G)} .\qedhere
    \]
\end{proof}

\begin{proof}[Proof of \Cref{thm:noauto}]
In the hyperfield lottery, the probability that $\Pib$ is a hyperfield which admits a non-trivial automorphism is at most 
$e^{-\Omega(\#G^2)}$, by \Cref{prop:noauto}. 
By \Cref{lm:modiso}, this gives 
\[
    \frac{\#\mathcal{A}(G, \epsilon)}{\#\mathcal{H}(G, \epsilon)} = 
    O\left(e^{-\Omega(\#G^2)} \cdot e^{(\log_2 n)^2}\right) = 
    e^{-\Omega(\# G^2)}. \qedhere
\]
\end{proof}
\begin{proof}[Proof of \Cref{thm:count}]
Since a pasture on $G$ is given by a unit $\epsilon \in G[2]$ and a subset of hexagons in $\hexagon(G)$, there are precisely 
\[
    N \deq \#G[2] \cdot 2^{\frac{1}{6}(\#G^2 + 3\#G + 2\#G[3])}
\]
pastures on $G$ by \Cref{prop:count}.
\begin{itemize}
    \item By \Cref{prop:prob}, 
    the proportion of these which are hyperfields  is 
    $1 - e^{-\Omega(\#G)}$. Hence the number of hyperfields on $G$ \emph{up to isomorphism} is at least 
    \[
        \frac{1}{\#\Aut G} \left(1 - e^{-\Omega(\#G)}\right) \cdot N
    \]
    \item By \Cref{prop:nonhyper}, the proportion of these which are non-hyperfields is at least 
    $e^{-O(\#G)}$, and the proportion that have a nontrivial automorphism is at most $e^{-\Omega(\#G^2)}$. Hence the number of hyperfields on $G$ \emph{up to isomorphism} is at most
    \[
       \pa{\frac{1}{\#\Aut G} \left(1 - e^{-O(\#G)}\right) + e^{-\Omega (\#G^2)}} \cdot N = \frac{1}{\#\Aut G} \left(1 - e^{-O(\#G)}\right) \cdot N
    \]
    for sufficiently large $\#G.$ \qedhere
\end{itemize}
\end{proof}
To prove \Cref{cor:asymptotic}, we will need an upper bound on the number of abelian groups of any given order. Though such a bound is easily obtained, we now discuss an upper bound on the number of \emph{groups} of given order, in preparation for the discussion in \Cref{sec:questions}.

\begin{Thm}[{\cite{Neumann69}, Theorem A}]\label[Thm]{thm:numgroups}
    The number of isomorphism classes of groups of order $n$ is 
    $e^{O((\log n)^4)}$.
\end{Thm}
This result by Neumann relies on the assumption that there is a similar bound on the number of finite simple groups of order $n$. The Classification of Finite Simple Groups shows that in fact this latter function is $O(1)$.
\begin{proof}[Proof of \Cref{cor:asymptotic}]
Let $G$ be any abelian group of order $n-1$.\footnote{By convention, hyperfields of order $n$ have underlying groups of order $n-1$. This matches the convention for finite fields.}
By \Cref{thm:count},
\[
    \log_2 \# \mathcal{H}(G) = 
    \log_2 \#G[2] - \log_2 \# \Aut G + 
    \frac{1}{6}((n-1)^2 + 3(n-1) + 2\#G[3]) - e^{-\Theta(\#G)}
\]
Since $\#G[2] = O(n)$, $\#G[3] = O(n)$ and
$\#\Aut G = O(n \log_2 n)$ (by \Cref{lm:numend}), we get
\[
    \log_2 \#\mathcal{H}(G) = \frac{n^2}{6} + O(n)
\]
Adding these up for all finite abelian groups of order $n-1$ using \Cref{thm:numgroups} gives
\[
    \log_2 \#\mathcal{H}_n = \frac{n^2}{6} + O((\log n)^4) + O(n) = \frac{n^2}{6} + O(n)
    \qedhere
\]
\end{proof}

\section{Beyond the commutative case}\label{sec:questions}
A number of difficulties arise in the attempt to extend our results to skew hyperfields. Before discussing these, we comment on a curious discrepancy between \Cref{conj:100}
and \Cref{cor:100}.
Whereas Baker and Jin's conjecture concerns quotients of fields, we have shown the a priori stronger result that hyperfields are almost never quotients of \emph{skew} fields. We do not know if this is a \emph{genuinely} stronger statement. That is:
\begin{Ques}\label[Ques]{ques:skew}
    Is every hyperfield which is the quotient of some skew field in fact the quotient of a field?
    Is this true if we insist that the hyperfield is finite?
\end{Ques}
The naive extension of 
\Cref{conj:100} to the setting of skew hyperfields holds for a kind of trivial reason, given by the following result.
\begin{Thm}\label[Thm]{thm:0skew}
     Let $\mathcal{S}_n$ be the set of isomorphism classes of skew hyperfields of order $n$, and let $\mathcal{H}_n \subseteq \mathcal{S}_n$ be the subset of those which are  hyperfields. Then
    \[
\frac{\#\mathcal{H}_n}{\#\mathcal{S}_n} = 1-e^{-\Omega(n^2)}.
    \]
\end{Thm}
We provide a proof of \Cref{thm:0skew} in \Cref{sec:skew}. The ingredients in the proof are twofold. 
The first is simply \Cref{thm:numgroups}, which gives us some control over the number of non-commutative groups of order $n$.
The second is the hexagon framework. It carries over to the non-commutative setting---a hexagon of $G$ being an orbit in $G^3$ under the action of $G \times S_3 \times G$ with the first $G$ acting by left multiplication and the second by right multiplication---with the caveat that hexagons can vary wildly in size and that the total number of hexagons may be quite small in comparison to the abelian case. This complicates any attempt to replicate the arguments in the proof of 
\Cref{prop:prob}.

\Cref{thm:0skew} states that asymptotically all but a vanishingly small proportion of finite skew hyperfields are in fact hyperfields.
Thus almost all skew hyperfields are non-quotient by \Cref{cor:100}.
It may be rightly objected that this answers the wrong question.
For instance, we may wish to restrict ourselves to certain classes of groups (e.g.\ non-commutative finite groups, symmetric groups, finite simple groups etc.). In the commutative case, \Cref{thm:main} provides a complete understanding of the generic behaviour of finite hyperfields with such restrictions. It is desirable to have a similar account for non-commutative groups.
One could also ask for an analogue of \Cref{thm:runner_up} for non-commutative groups.
An analogue of \Cref{thm:noauto} would say that for a generic skew hyperfield on a finite group $G$, all automorphisms are induced by inner automorphisms of $G$.
Finally, it would be interesting to have asymptotics for the number of isomorphism classes of skew hyperfields on a given non-commutative group.

\section*{Acknowledgements}
The authors wish to thank Noah Solomon for bringing the paper \cite{hobby} to our attention, as well as the \emph{Arizona-New Mexico Symposium on Commutative Algebra and its Interactions} for having provided a forum conducive to cooperative discussion. It was there that we met Noah. We are also grateful to Alexander Divoux for having spotted an error in an earlier version of this manuscript.
\printbibliography

\appendix
\section{Products of hyperfields}\label[appendix]{sec:appendix}
Here we discuss products in the categories of pastures and hyperfields in relation to the $0/0$ property. 
The category of pastures is complete and cocomplete, as shown in \cite{creech}.
It follows from the results there that the product of the pastures $\mathbb{P}_1$ and $\mathbb{P}_2$ has underlying group $\mathbb{P}_1^{\times} \times \mathbb{P}_2^{\times}$, unit $\begin{bmatrix} -1\\ -1 \end{bmatrix}$ and hexagons given by
\[
    \Hex(\mathbb{P}_1 \times \mathbb{P}_2) = 
    \left\{\overline{\left(\begin{bmatrix} u_1\\ u_2\end{bmatrix}, \begin{bmatrix}v_1\\ v_2\end{bmatrix}\right)} \;\bigg|\; (u_1, v_1) \in \Hex(\mathbb{P}_1), 
    (u_2, v_2) \in \Hex(\mathbb{P}_2) \right\}
\]
(Here and below, we write 
$\begin{bmatrix}a\\ b\end{bmatrix}$ for the element $(a, b)$ in a product of two groups. This has the advantage of being visually distinct from our notation for fundamental pairs.)

If $\mathbb{H}$ and $\mathbb{G}$ are hyperfields, it is natural to ask when the product 
$\mathbb{H} \times \mathbb{G}$ \emph{in the category of pastures} is again a hyperfield.
If this is so, then by \Cref{prop:fullyfaithful} this product is also their product in the category of hyperfields. We shall prove the following.
\begin{Thm}\label[Thm]{thm:pastureproduct}
    Let $\Hb$ and $\G$ be hyperfields. Then the pasture product $\mathbb{H} \times \mathbb{G}$ of $\Hb$ and $\G$ is a hyperfield if and only if 
    \begin{itemize}
        \item both $\Hb$ and $\G$ are $0/0$, or
        \item $\Hb\cong \K$ or $\G\cong \K$, or
        \item $\Hb\cong \G\cong \F_2.$
    \end{itemize}
    In the first of these cases, the product is again $0/0$.
\end{Thm}
To prove \Cref{thm:pastureproduct}, we will need two lemmas.
\begin{Lm}\label[Lm]{lm:distinctpairs}
    Let $\mathbb{H}$ be a hyperfield with $\mathbb{H} \ncong \mathbb{F}_2, \mathbb{F}_3, \mathbb{K}$. Then there exist $x,y \in \mathbb{H}^{\times}$ with $x \neq y$ such that 
    $\overline{(x,y)} \in \Hex(\mathbb{H})$.
\end{Lm}
\begin{proof}
    Since $\mathbb{H}$ is not $\mathbb{F}_2$, it must have some fundamental pair.
    If $\mathbb{H}$ has a fundamental pair
    $(x,y) \neq (1,1)$ then either $x \neq y$---in which case we are done---or else $x = y$, in which case both $(1, x^{-1})$
    and $(1,y^{-1})$ are fundamental pairs in $\mathbb{H}$ and at least one of these has the desired form.
    Suppose $(1,1)$ is the only fundamental pair in $\Hex(\mathbb{H})$. By criterion (A) in \Cref{prop:hexagon_characterization}, any element in $\mathbb{H} \setminus \{0,-1\}$ lies in some fundamental pair. Thus $\mathbb{H} = \{0,1,-1\}$ and its only fundamental pair is $(1,1)$.
    Depending on whether $1 = -1$ or not, this gives $\mathbb{H} \cong \mathbb{K}$ or $\mathbb{F}_3$.
\end{proof}
The following is a variant on \Cref{prop:hexagon_characterization}.
\begin{Lm}\label[Lm]{lm:0/0characterization}
    A pasture $\Pb$ on a group $G$ arises from a 0/0 hyperfield if and only if the following two conditions hold:
    \begin{enumerate}[(A)]
        \item For all $x \in G$ there exists $y \in G$
        such that ${(x,y)} \in \Hex(\mathbb{P})$.
        \item For all $(x, y),{(z, w)} \in \Hex(\Pb)$ there exists $t \in G$ such that ${(tx, -tz)}, {(-ty, tw)} \in \Hex(\Pb)$.
    \end{enumerate}
\end{Lm}
\begin{proof}
We compare each condition to its analogue in \Cref{prop:hexagon_characterization}. For condition (A), the strengthening is that we allow $x = -1$. But this is true for any non-field hyperfield by \Cref{rmk:norestriction}. It is clear that fields cannot be 0/0. For condition (B), the strengthening is precisely that of allowing $(x,y) = (z,w)$. That is, we are insisting that if $(x,y) \in \Hex(\mathbb{P})$ then there exists $t \in G$ such that $(tx, -tx), (-ty, ty) \in \Hex(\mathbb{P})$. In hyperfield terms, this is the statement that 
$1 \boxplus -1 \ni t^{-1} x^{-1}, t^{-1} y^{-1}$.  
This is precisely the statement that there exist $r, s \in 1 \boxplus -1 \setminus \{0\}$ such that $rs^{-1} = xy^{-1}$. This must hold for all fundamental pairs $(x,y)$.
Using the involution $(x, y) \mapsto (xy^{-1}, y^{-1})$ on fundamental pairs, we see that the condition becomes that whenever $(x, y) \in \Hex(\mathbb{P})$ then there exist $r,s \in 1 \boxplus -1 \setminus \{0\}$ such that 
\[
    rs^{-1} = (xy^{-1})(y^{-1})^{-1} = x
\]
Assuming (A) holds, every element of $\mathbb{P}^{\times}$ lies in some fundamental pair. Thus we have recovered precisely the $0/0$ condition (apart from the trivial case $x = 0$).
\end{proof}

\begin{proof}[Proof of \Cref{thm:pastureproduct}]
    If $\mathbb{H} \cong \mathbb{F}_2$ then $\mathbb{H} \times \mathbb{G}$ has cardinality 
    equal to $\#\mathbb{G}$ but no hexagons (because $\mathbb{F}_2$ doesn't have any). By \Cref{lm:numhex}, this forces 
    $\#\mathbb{G} \leq 2$. This implies that $\mathbb{G}$ is isomorphic to one of $\mathbb{F}_2$ or $\mathbb{K}$. 
    Suppose now that neither of $\mathbb{H}, \mathbb{G}$
    is isomorphic to either of
    $\mathbb{K}, \mathbb{F}_2$, but that $\mathbb{H}$ is isomorphic to $\mathbb{F}_3$. Let $x \in \mathbb{G}$ satisfy $x \neq 0, -1$. Since in $\mathbb{F}_3$ the element $-1$ does not lie in any hexagon, the element 
    $\begin{bmatrix}
        -1 \\ x
    \end{bmatrix}$ does not lie in any hexagon of $\mathbb{H} \times \mathbb{G}$ but is not the unit of this pasture. By criterion (A) of \Cref{prop:hexagon_characterization}, this shows that $\mathbb{H} \times \mathbb{G}$ is not a hyperfield.

    Now suppose that $\mathbb{H}, \mathbb{G}$ are not isomorphic to any of 
    $\mathbb{K}, \mathbb{F}_2, \mathbb{F}_3,$ but their product is a hyperfield.
    We will show that they must be $0/0$. It suffices to do this for $\mathbb{G}$.
    By \Cref{lm:distinctpairs}, we can find $(a,b) \in \Hex(\mathbb{H})$ such that $a \neq b$.
    Let $(c,d) \in \Hex(\mathbb{G})$ be any fundamental pair.
    Then
    \[
       \left(
       \begin{bmatrix}
           a \\ c
       \end{bmatrix},
        \begin{bmatrix}
           b \\ d
       \end{bmatrix}
       \right)
       \quad \text{and} \quad
       \left(
       \begin{bmatrix}
           b \\ c
        \end{bmatrix},
        \begin{bmatrix}
           a \\ d
       \end{bmatrix}
       \right)
    \]
    are distinct fundamental pairs of 
    $\mathbb{H} \times \mathbb{G}$.
    By criterion (B) in \Cref{prop:hexagon_characterization}, we can find $t \in \mathbb{H}^{\times}, u \in \mathbb{G}^{\times}$ such that 
    \[
       \left(
       \begin{bmatrix}
           ta \\ uc
       \end{bmatrix},
        \begin{bmatrix}
           -tb \\ -uc
       \end{bmatrix}
       \right)
       \quad \text{and} \quad
       \left(
       \begin{bmatrix}
           -tb \\ -ud
        \end{bmatrix},
        \begin{bmatrix}
           ta \\ ud
       \end{bmatrix}
       \right)
    \]
    are both fundamental pairs of $\mathbb{H} \times \mathbb{G}$. Looking just at the bottom row, the condition is that 
    $u^{-1} c^{-1}$ and 
    $u^{-1}d^{-1}$ both lie in 
    the sum $1 \boxplus -1$ in $\mathbb{G}$. In particular, in $\mathbb{G}$ we can find $r, s \in 1 \boxplus -1$ with $s \neq 0$ such that $rs^{-1} = cd^{-1}$. Note that $0 \neq s\in 1\bp -1$ implies that $\mathbb{G}$ is not a field. Now let $x \in \mathbb{G}^{\times}$ be arbitrary. By criterion (A) of \Cref{prop:hexagon_characterization} and \Cref{rmk:norestriction}, there is some fundamental pair $(x,y) \in \Hex(\mathbb{G})$.
    Then $(xy^{-1}, y^{-1})$ is also a fundamental pair of $\mathbb{G}$. Via the identifications 
    $c \rightsquigarrow xy^{-1}$, $d \rightsquigarrow y^{-1}$, we see that for some $r, s \in 1 \boxplus -1$ we have 
    $rs^{-1} = x$. This shows that the 0/0 condition holds (aside from the trivial case $x=0$).
    
    It remains to show that the products in the announced cases are actually hyperfields.
    The pasture product isomorphisms
    $\mathbb{F}_2 \times \mathbb{F}_2 \cong \mathbb{F}_2$ 
    and $\mathbb{P} \times \mathbb{K} \cong \mathbb{P}$ (for all pastures $\mathbb{P}$) are easily verified. This shows that we get a hyperfield in the second and third cases.
    
    Suppose now that $\mathbb{H}, \mathbb{G}$ are 0/0 hyperfields. We show that $\mathbb{H} \times \mathbb{G}$ is too.
    First we check that it satisfies criterion (A) in \Cref{lm:0/0characterization}. 
    Let $\begin{bmatrix} h\\ g\end{bmatrix} \in \mathbb{H}^{\times} \times \mathbb{G}^{\times}$.
    By the same property for $\mathbb{H}$ and $\mathbb{G}$, we can find $u \in \mathbb{H}^{\times}$, $v \in \mathbb{G}^{\times}$ such that 
    $(h, y) \in \Hex(\mathbb{H})$
    and $(g,z) \in \Hex(\mathbb{G})$.
    Then 
    \[
        \left(
        \begin{bmatrix}
        h \\ g
        \end{bmatrix},
        \begin{bmatrix}
        y \\ z
        \end{bmatrix}
        \right) \in \Hex(\mathbb{H} \times \mathbb{G})
    \]
    For property (B), suppose 
     \[
       \left(
       \begin{bmatrix}
           x_1 \\ x_2
       \end{bmatrix},
        \begin{bmatrix}
           y_1 \\ y_2
       \end{bmatrix}
       \right)
       \quad \text{and} \quad
       \left(
       \begin{bmatrix}
           z_1 \\ z_2
        \end{bmatrix},
        \begin{bmatrix}
           w_1 \\ w_2
       \end{bmatrix}
       \right)
    \]
    are fundamental pairs of $\mathbb{G} \times \mathbb{H}$. Then 
    $(x_i, y_i), (z_i, w_i)$ are fundamental pairs of 
    $\mathbb{H}$ and $\mathbb{G}$ for $i = 1, 2$ respectively. By property (B) for these hyperfields, we can find $t_1 \in \mathbb{H}^{\times}$ and 
    $t_2 \in \mathbb{G}^{\times}$ such that 
    $(t_i x_i, -t_i z_i), (-t_i y_i, t_i w_i)$ are fundamental pairs in the appropriate hyperfields.
    Thus,
    \[
       \left(
       \begin{bmatrix}
           t_1 \\ t_2
       \end{bmatrix}\cdot
       \begin{bmatrix}
           x_1 \\ x_2
       \end{bmatrix},
       -\begin{bmatrix}
           t_1 \\ t_2
       \end{bmatrix}\cdot
        \begin{bmatrix}
           z_1 \\ z_2
       \end{bmatrix}
       \right)
       \quad \text{and} \quad
       \left(
       -\begin{bmatrix}
           t_1 \\ t_2
       \end{bmatrix}\cdot
       \begin{bmatrix}
           y_1 \\ y_2
        \end{bmatrix},
        \begin{bmatrix}
           t_1 \\ t_2
       \end{bmatrix}\cdot
        \begin{bmatrix}
           w_1 \\ w_2
       \end{bmatrix}
       \right)
    \]
    are fundamental pairs in $\mathbb{H} \times \mathbb{G}$. This gives property (B) for the product.
\end{proof}
\begin{Rmk}
    \Cref{thm:pastureproduct} also holds with ``pasture'' replaced with ``tract'' in the sense of \cite{baker19}, though we shall not comment further on this here.
\end{Rmk}
In contradistinction to the above result, we observe that the product of two hyperfields may exist in the category of hyperfields without coinciding with the pasture product.
To see this, we need the following proposition.
\begin{Prop}\label[Prop]{prop:HtoF}
    Let $\mathbb{H}$ be a (skew) hyperfield and $\mathbb{F}$ a (skew) field. If a homomorphism 
    $\mathbb{H} \to \mathbb{F}$ exists, then $\mathbb{H}$ is a (skew) field.
\end{Prop}
\begin{proof}
    Suppose that $\mathbb{H}$ is not a (skew) field.
    By \Cref{rmk:check1-1},
    the sum $1 \boxplus -1$ in $\mathbb{H}$ contains some nonzero element. If there were a morphism $\mathbb{H} \to \mathbb{F}$, the sum $1 \boxplus -1$ in 
    $\mathbb{F}$ would have to contain a nonzero element.
\end{proof}

From this observation, we deduce:
\begin{Prop}
    Let $p$ be a prime number and 
    $\mathbb{F}_p$ the corresponding prime field. 
    Then the product of $\mathbb{F}_p$ with itself in the category of hyperfields is again $\mathbb{F}_p$.
\end{Prop}
\begin{proof}
    By \Cref{prop:HtoF}, any homomorphism into $\mathbb{F}_p$ has domain a field. Since any field homomorphism is an embedding and since $\mathbb{F}_p$ is prime, the only hyperfield homomorphism into $\mathbb{F}_p$ is the identity
    $\mathbb{F}_p \to \mathbb{F}_p$. From this, the claim follows.
\end{proof}
Note that no field is ever 0/0.

On the other hand, categorical products may fail to exist in the category of hyperfields.
\begin{Prop}
    Let $p, q$ be relatively prime \emph{prime powers}.
    Then the product of 
    $\mathbb{F}_p$ and 
    $\mathbb{F}_q$ does not exist in the category of hyperfields.
\end{Prop}
\begin{proof}
    As above, we note that any hyperfield homomorphism into either of $\mathbb{F}_p$ or $\mathbb{F}_q$ has domain a field by \Cref{prop:HtoF}.
    Since a homomorphism of fields is only possible within a fixed characteristic, there exists no hyperfield which maps into both $\mathbb{F}_p$ and $\mathbb{F}_q$. It follows that these two cannot have a product in the hyperfield category.
\end{proof}

\section{The sparsity of finite skew hyperfields}\label[appendix]{sec:skew}
Here we sketch a proof of \Cref{thm:0skew}. The proof proceeds by mimicking the hexagon framework of \Cref{sec:notation}, with subtle but important differences.

The following is the analogue of \Cref{df:hexagon}.
\begin{Df}\label[Df]{df:skewhexagon}
    Let $G$ be a group. The combined effects of
    left and right multiplication by $G$ on $G^3$ and the permutation action of $S_3$ on $G^3$ may be encapsulated in the action of the group $G \times S_3 \times G$ on $G^3$. An orbit of this action is called a \emph{hexagon} of $G$.
    We write $\hexagon(G)$ for the collection of hexagons of $G$.
    Using the normalization map 
    $(x, y, z) \mapsto (xz^{-1}, yz^{-1})$, we may identify $\hexagon(G)$ with the set of orbits of $G^2$ under the action of $S_3 \times G$---where $G$ acts by conjugation. We write 
    $\overline{(x,y)}$ for the hexagon thus associated to the pair $(x,y) \in G^2$.
\end{Df}
As in the commutative case, it follows from axioms (I) and (III) in \Cref{df:hyperfield} that for a skew hyperfield $\mathbb{H}$ the set 
\[
    \{(a,b,c) \in \mathbb{H}^{\times} \mid 0 \in a \boxplus b \boxplus c\}
\]
is invariant under the left and right multiplicative actions of 
$(\mathbb{H}^{\times})^3$ as well as the permutation action of $S_3$. 
We ought therefore to consider the corresponding set of hexagons of 
$\mathbb{H}^{\times}$.
These will be called the hexagons of $\mathbb{H}$.
This data may be encapsulated as follows (c.f.\ \Cref{df:pasture}).
\begin{Df}
    A \emph{skew pasture} $\mathbb{P}$ consists of a triple $\mathbb{P} = (G, \epsilon, \Hex(\mathbb{P}))$ where 
    \begin{itemize}
        \item $G$ is a group, called the \emph{underlying group} of $\mathbb{P}$ and denoted 
        $\mathbb{P}^{\times}$.
        \item $\epsilon \in G$ is an element in the center of $G$ satisfying $\epsilon^2 = 1$, called the \emph{unit} of $\mathbb{P}$.
        \item $\Hex(\mathbb{P}) \subseteq \hexagon(G)$ is a subset of hexagons of $G$, called the \emph{nullset} of $\mathbb{P}$.
    \end{itemize}
    A morphism of skew pastures is defined as in \Cref{df:pasturemorphism}.
\end{Df}
As we remarked above, every skew hyperfield determines a pasture consisting of (i) its multiplicative group, (ii) its unit and (iii) its collection of hexagons. 
We omit the straightforward proof of the following proposition.
\begin{Prop}\label[Prop]{prop:skewfullyfaithful}
    A skew hyperfield is determined by its skew pasture. 
    A homomorphism of hyperfields induces a morphism of the corresponding skew pastures---and every morphism between these pastures arises in this way.
    This induces a fully faithful embedding of the category of skew hyperfields into the category of skew pastures.
\end{Prop}
The key observation in the proof of \Cref{thm:0skew} is the following.
\begin{Prop}\label[Prop]{prop:5/8}
    Let $G$ be a non-commutative 
    group of order $n$. Then $G$ has at most 
    $\frac{1}{6}(5n^2/8 + 5n)$ hexagons.
\end{Prop}
\begin{proof}
We use Burnside's Lemma to count the number of orbits of $G^2$ under the action of $S_3 \times G$ as in \Cref{df:skewhexagon}. For $g, h \in G$, we write ${}^gh$ for the conjugate $ghg^{-1}$.
The fixed points of $(1,g) \in S_3 \times G$ are precisely the elements 
${(x,y)} \in G^2$ with $x,y \in C(g)$ (where $C(g)$ is the centralizer of $g$).
If $\tau$ is the transposition in $S_3$ that sends $(x,y) \mapsto (y,x)$, then 
any pair $(x,y) \in G^2$ fixed by
$(\tau, g)$ must satisfy $y = {}^gx$.
There are therefore at most $n$ such pairs. The same is true for any other transposition. Finally, if $\rho$ is the 3-cycle that sends $(x,y) \mapsto (yx^{-1}, x^{-1})$ then any pair $(x,y) \in G^2$ fixed by $(\rho, g)$ must satisfy $y = {}^gx^{-1}$. Again, there are at most $n$ of these. The same is true for the inverse 3-cycle.

From this discussion and Burnside's Lemma, it follows that an upper bound for the number of orbits is 
\begin{equation}\label{eq:burnside}
    \#\hexagon(G) \leq \frac{1}{6n}\sum_{g  \in G} ((\#C(g))^2 + 5n) = \frac{1}{6}\left(\frac{1}{n} \sum_{g \in G} \#C(g)^2 + 5n\right)
\end{equation}
Write $Z(G)$ for the center of $G$. We have
$C(g) = G$ for $g \in Z(G)$. For $g \notin Z(G)$, $C(g)$ is a proper subgroup of $G$, so has index at least 2. 
Note that $Z(G)$ itself is a proper subgroup of $G$, so also has index at least 2.
From this we conclude
\begin{multline*}
    \sum_{g \in G} \#C(g)^2 = \sum_{g \in Z(G)} \#C(g)^2 + \sum_{g \notin Z(g)} \#C(g)^2
    \leq n^2 \cdot \#Z(G) + \frac{n^2}{4} (n - \#Z(G))\\
    = \frac{3n^2}{4} \cdot \#Z(G) + \frac{n^3}{4} \leq \frac{3n^2}{4} \cdot \frac{n}{2} + \frac{n^3}{4} = \frac{5n^3}{8}
\end{multline*}
Putting this into (\ref{eq:burnside}) gives 
\[
    \#\hexagon(G) \leq \frac{1}{6}\left(\frac{5n^2}{8} + 5n\right)
    \qedhere
\]
\end{proof}
\begin{proof}[Proof of \Cref{thm:0skew}]
    Let $G$ be a finite nonabelian group of order $n-1$.
    By \Cref{prop:skewfullyfaithful}, a skew hyperfield on $G$ is determined by the data of (i) an element in the center of $G$ of order at most 2 and (ii) a subset of the hexagons of $G$. The number of skew hyperfields on $G$ is therefore bounded above by
    \[
        (n-1) \cdot 2^{\#\hexagon(G)} \leq 
        n \cdot 2^{\frac{1}{6}(5(n-1)^2/8 + 5(n-1))} = 2^{n^2/6} \cdot  2^{-n^2/16 + O(n)}
    \]
    using \Cref{prop:5/8}.
    Since there are at most $e^{O((\log n)^4)}$ groups of order $n-1$ by \Cref{thm:numgroups} and 
    $2^{n^2/6 + O(n)}$ hyperfields of order $n$ by \Cref{cor:asymptotic}, we have
    \[
        \frac{\#S_n - \#H_n}{\#H_n}
        \leq \frac{e^{O((\log n)^4)} \cdot 2^{n^2/6} \cdot  2^{-n^2/16 + O(n)}}{2^{n^2/6}} = 2^{-n^2/16+O(n)}
    \]
    from which the claim follows.
\end{proof}
\end{document}